\documentclass[12pt]{article}
\usepackage{amsmath,amssymb,amsthm}
\usepackage[margin=1in]{geometry}
\usepackage{subcaption}
\usepackage{tablefootnote}
\usepackage[hidelinks]{hyperref}
\usepackage[backend=bibtex,giveninits=true,maxbibnames=99]{biblatex}
\newtheorem{theorem}{Theorem}
\newtheorem{lemma}[theorem]{Lemma}
\newtheorem{fact}[theorem]{Fact}

\newcommand{\A}{\mathcal{A}}
\newcommand{\E}[1]{\mathbb{E}\left[#1\right]}

\title{Upper bounds on the average edit distance between two random strings}
\author{Matthieu Rosenfeld\\
\small LIRMM, Univ Montpellier, CNRS, Montpellier, France}
\date{\today}

\begin{document}
\maketitle

\begin{abstract}
We study the average edit distance between two random strings. More precisely, we adapt a technique introduced by Lueker in the context of the average longest common subsequence of two random strings to improve the known upper bound on the average edit distance. We improve all the known upper bounds for small alphabets. We also provide a new implementation of Lueker technique to improve the lower bound on the average length of the longest common subsequence of two random strings for all small alphabets of size other than $2$ and $4$.
\end{abstract}
\section{Introduction}
We consider two different notions of similarity between pairs of strings.
The first notion is the \emph{length of the longest common subsequence} between two strings $u$, and $v$, denoted by $\operatorname{LCS}(u,v)$. 
The second notion is the edit distance (also called Levenshtein distance) between two strings $u$, and $v$, denoted by $d_e(u,v)$, which is the minimum number of substitutions, deletions and insertions necessary to transform $u$ to $v$.

In particular, we are interested in the average edit distance denoted $e_k(n)$ (resp. average length of the longest common subsequence denoted $\ell_k(n)$) between two random $k$-ary strings of length $n$. More precisely, we provide bounds on the two following quantities whose existence is implied by Fekete's Lemma
\begin{equation*}
\alpha_k=\lim\limits_{n\rightarrow\infty}\frac{e_k(n)}{n}\text{ and } \gamma_k=\lim\limits_{n\rightarrow\infty}\frac{\ell_k(n)}{n}\,.
\end{equation*}
The so called Chvátal-Sankoff constant $\gamma_2$, and the other $\gamma_k$ received a lot of attention since 1975 \cite[\ldots]{10.1214/aoap/1177004903,Baeza-Yates1999Aug, Bukh2022Apr,Bundschuh2001Aug,chvatal,Dancik2005May,Dancik1994,Deken1979Jan,Kiwi2005Nov,Tiskin2022Dec}. In particular, we know that $\gamma_2$ is algebraic \cite{Tiskin2022Dec} and that $0.788071 \le\gamma_2\le0.826280$ \cite{Lueker2009May}. Until recently the best bounds on $\gamma_k$ for other small values of $k$ were given in \cite{Kiwi2009Jul}. A few days before the publication of the present article on the arXiv another preprint improved most of the bounds known for $\gamma_k$ and variant of $\gamma_k$ that consider the expected LCS of $t$ random strings \cite{Heineman2024Jul}. In particular, they proved $\gamma_2\ge0.792665$
We improve on all the bounds for $\gamma_k$ given by \cite{Heineman2024Jul} other than $\gamma_2$ and $\gamma_4$ where we only beat the previous best bound by Lueker \cite{Lueker2009May}.

The study of $\alpha_k$ ($\alpha_4$ and $\alpha_2$ in particular) only started more recently in the context of DNA reconstruction from reads with errors \cite{Ganguly2021Mar} and nearest neighbor search \cite{Rubinstein2018Jun}. The best bounds known on the $\alpha_k$ have been obtained recently \cite{Bilardi2022Nov,Schimd2019Oct}.

The main contribution of this article is to improve the upper bounds on the value of $\alpha_k$ for small $k$ (see table \ref{bak}).  For this, we adapt the technique that Lueker introduced to lower bounds $\gamma_2$ \cite{Lueker2009May}. We emphasize the fact that our proof relies on adapting the ideas of \cite{Lueker2009May}. The proof requires the computation of a large vector obtained by the iteration of a particular transformation. We also provide another implementation of the technique used by Lueker for $\gamma_k$ (which is also what was recently done in \cite{Heineman2024Jul}). We improve the bounds for all considered values of $\gamma_k$ other than $\gamma_2$ and $\gamma_4$ (see table \ref{bgk}).

\begin{table}[h]\begin{subtable}{0.48\linewidth}
    \footnotesize
    \begin{tabular}{cccc}\hline
        & \multicolumn{2}{c}{Upper bounds}  & Lower bounds\\
 k & Our results & from \cite{Bilardi2022Nov} & from \cite{Bilardi2022Nov}\\\hline\hline
 2 &\bfseries 0.315514 & 0.36932 & 0.17372\cite{Lueker2009May}\\\hline
 3 &\bfseries 0.47276 & 0.53426 & 0.28366\\\hline
 4 &\bfseries 0.56578 & 0.63182 & 0.35978\\\hline
 5 &\bfseries 0.6325 & 0.70197 & 0.41517\\\hline
 6 &\bfseries 0.68424 & 0.75149 & 0.45776\\\hline
 7 &\bfseries 0.72016 & 0.79031 & 0.49183\\\hline
 8 &\bfseries 0.74896 & 0.81166 & 0.51990\\\hline
 9 &\bfseries 0.77264 &  & \\\hline 
 10 &\bfseries 0.7925 &  &\\\hline
 11 &\bfseries 0.8095 &  &\\\hline
 12 &\bfseries 0.82432 &  &\\\hline
 13 &\bfseries 0.83744 &  &\\\hline
 14 &\bfseries 0.84646 &  &\\\hline
 15 &\bfseries 0.85608 &  &\\\hline
 16 &\bfseries 0.86462 & 0.89554 & 0.64475\\\hline
 17 &\bfseries 0.87228 &  &\\\hline
 18 &\bfseries 0.87916 &  &\\\hline
 19 &\bfseries 0.88536 &  &\\\hline
 20 &\bfseries 0.89102 &  &\\\hline
 21 &\bfseries 0.89614 &   &\\\hline
 22 &\bfseries 0.90084 &  &\\\hline
 23 &\bfseries 0.90514 &  &\\\hline
 24 &\bfseries 0.90912 &  &\\\hline
 25 &\bfseries 0.91278 &  &\\\hline
 32 &\bfseries 0.93228 & 0.96588 & 0.73867\\\hline
 100 &\bfseries 0.97946 & &\\\hline
 100 &\bfseries 0.97982 & &\\\hline
    \end{tabular}
    \caption{New upper bounds on  $\alpha_k$}\label{bak}
\end{subtable}
\begin{subtable}{0.55\linewidth}
    \footnotesize
    \begin{tabular}{cccc}\hline
         & \multicolumn{2}{c}{Lower bounds}  & Upper bounds\\
 k & Our results & Previous best  & from \cite{Dancik1994}\\\hline\hline
 2 & 0.789872& \bfseries 0.792665 \cite{Heineman2024Jul} & 0.8263\cite{Lueker2009May}\\\hline
 3 & \bfseries 0.68422 & 0.682218 \cite{Heineman2024Jul}&0.76581\\\hline
 4 & 0.61422 & \bfseries 0.614333 \cite{Heineman2024Jul}&0.70824\\\hline
 5 & \bfseries 0.56206 & 0.549817 \cite{Heineman2024Jul}&0.66443\\\hline
 6 & \bfseries 0.51850 &0.4992299 \cite{Heineman2024Jul}&0.62932\\\hline
 7 & \bfseries 0.48712 & 0.466481 \cite{Heineman2024Jul}&0.60019\\\hline
 8 & \bfseries 0.46074 & 0.438799 \cite{Heineman2024Jul}&0.57541\\\hline
 9 & \bfseries 0.43806 & 0.414876 \cite{Heineman2024Jul}&0.55394\\\hline
 10 & \bfseries 0.41826 & 0.393811 \cite{Heineman2024Jul}&0.53486\\\hline
 11 & \bfseries 0.40072 & 0.37196 \cite{Dancik1994}&0.51785\\\hline
 12 & \bfseries 0.38504 & 0.35899 \cite{Dancik1994}&0.50260\\\hline
 13 & \bfseries 0.37088 & 0.34737 \cite{Dancik1994}&0.48880\\\hline
 14 & \bfseries 0.35798 & 0.33687 \cite{Dancik1994}&0.47620\\\hline
 15 & \bfseries 0.34616 & 0.32732 \cite{Dancik1994}&0.46462\\\hline
 16 & \bfseries 0.33528 & & \\\hline
 17 & \bfseries 0.32518 & & \\\hline
 18 & \bfseries 0.31580 & & \\\hline
 19 & \bfseries 0.30702 & & \\\hline
 20 & \bfseries 0.29880 & & \\\hline
 21 & \bfseries 0.29106 & & \\\hline
 22 & \bfseries 0.28376 & & \\\hline
 23 & \bfseries 0.27686 & & \\\hline
 24 & \bfseries 0.27032 & &  \\\hline
 25 & \bfseries 0.26412 & & \\\hline
 50 & \bfseries 0.16930 & & \\\hline
 100 & \bfseries 0.0991 & & \\\hline
 1000 &\bfseries 0.01164 & & \\\hline
    \end{tabular}
    \caption{New lower bounds on $\gamma_k$}\label{bgk}
\end{subtable}
\caption{A summary of the bounds provided in this article. We indicate in bold the best results.}
\end{table}

\section{Average edit distance}

For any strings $u,v$, $d_e(u,v)$ is the \emph{edit distance} (or \emph{Levenshtein distance}) between $u$ and $v$, that is, the minimum number of substitutions, deletions and insertions necessary to transform $u$ to $v$.
Given two strings $u,v\in \A^*$, and letters $a,b\in \A$,
we have 
\begin{equation}
d_e(au,bv)=
\begin{cases} 
 d_e(u,v) & \text{ if } a=b \\
    1+\min(d_e(u,v),d_e(u,bv),d_e(au,v)) & \text{ otherwise }
 \end{cases}
\end{equation}
and for all $u$, $d_e(u,\varepsilon)=d_e(\varepsilon,u)=|u|$, where $\varepsilon$ is the empty string.
For any strings $u,v,u',v'$, we have 
\begin{equation}\label{decomposition}
 d_e(uv,u'v')\le d_e(u,u')+d_e(v,v')\,.
\end{equation} 

For any $n$, we let $U_{n,k}= d_e(X_1\ldots X_n, Y_1\ldots, Y_n)$
where the $X_i$ and $Y_i$ are taken uniformly at random from $\A=\{1,\ldots, k\}$. 
For all $k$, the quantity $\alpha_k$ is defined as
\begin{equation*}
    \alpha_k = \lim_{n\rightarrow\infty}\frac{\E{d_e(X_1\ldots X_n, Y_1\ldots Y_n)}}{n}= \lim_{n\rightarrow\infty}\frac{U_{n,k}}{n}\,.
\end{equation*}
Linearity of expectation and equation \eqref{decomposition} imply that the sequence $\left(\E{U_{n,k}}\right)_{n\ge1}$ is subadditive. Thus by Feketes's lemma, $\alpha_k$ is well-defined. 

Since we will always be working with one specific ambient alphabet $\A$, we omit the $k$, and, from now on, we write for instance $U_n$ instead of $U_{n,k}$. Similarly, in the following the $X_i$ and $Y_i$ always denote random variables taken uniformly at random from $\A=\{1,\ldots, k\}$. 
For all $1\le i,j,\ell,m \le n$, we let 
\begin{equation*}
 U_{[i,j],[\ell,m]} = d_e(X_i\ldots X_j, Y_\ell\ldots, Y_m)\,.
\end{equation*}    
Finally, we let $V_n=\min\limits_{i\in\{0,\ldots, n\}}U_{[1,i],[1,n-i]}$ and
\begin{equation*}
    \alpha'_k=\lim_{n\rightarrow\infty}\frac{\E{V_n}}{n}\,.
\end{equation*}
We will show that $\alpha_k = 2\alpha'_k$, and then explain how we compute bounds on $\alpha'_k$.

\subsection{From \texorpdfstring{$\alpha_k$}{} to \texorpdfstring{$\alpha'_k$}{}}
We first prove that $\alpha_k = 2\alpha'_k$. The proof is directly inspired by the similar proof for $\gamma_k$ in \cite{10.1214/aoap/1177004903}.
We will use the following special case of McDiarmid's inequality \cite{BibEntry1989Aug}.
\begin{theorem}[McDiarmid's inequality]
Let $f: \mathcal{X}_1\times \mathcal{X}_2\times\cdots\times \mathcal{X}_m\rightarrow\mathbb{R}$ be a function with the property that changing any one argument of $f$ while holding the others fixed changes the value of $f$ by at most $c$. 
Consider independent random variables $X_1,\ldots, X_m$ where $X_i\in\mathcal{X}_i$ for all $i$. Then, for any $\varepsilon>0$,
\begin{equation*}
 \mathbb{P}\left[f(X_1,\ldots, X_m)-\E{f(X_1,\ldots, X_m)}\ge\varepsilon\right]\le \operatorname{exp}\left(-\frac{2\varepsilon^2}{mc^2}\right)\,, 
\end{equation*}
and
\begin{equation*}
 \mathbb{P}\left[f(X_1,\ldots, X_m)-\E{f(X_1,\ldots, X_m)}\le-\varepsilon\right]\le \operatorname{exp}\left(-\frac{2\varepsilon^2}{mc^2}\right)\,.
\end{equation*}
\end{theorem}

We can apply this theorem to both $V_n$ and $U_n$ with $m=2n$, $c=1$ and well-chosen values of $\varepsilon$ that are going to be useful in what follows.
\begin{lemma}
For all $n$,
\begin{equation}\label{concVn}
 \mathbb{P}\left[V_n\le\E{V_n}+\sqrt{n\log 2}\right]\ge \frac{1}{2}\,,
\end{equation}
and
\begin{equation}\label{concUn}
\mathbb{P}\left[U_n\ge\E{U_n}+\sqrt{n\log (4(n+1)^2)}\right]\le \frac{1}{4(n+1)^2}\,.
\end{equation}
\end{lemma}

\begin{lemma}
 For all $k$,
    $$\alpha_k= 2\alpha'_k\,.$$
\end{lemma}
\begin{proof}
 By definition, for all $n$, 
    \begin{align*}
 \frac{\E{V_n}}{n}
 \le\frac{\E{U_{[1,\lceil n/2\rceil],[1,\lfloor n/2\rfloor]}}}{n}
 \le\frac{\E{U_{[1,\lceil n/2\rceil],[1,\lceil n/2\rceil]}}}{\lceil n/2\rceil}\cdot\frac{\lceil n/2\rceil}{n}\,.
    \end{align*}
 The limit of the right-hand side as $n$ goes to infinity is $\frac{\alpha_k}{2}$ which implies $\alpha'_k\le\frac{\alpha_k}{2}$. 

 We now focus on proving  $\alpha'_k\ge\frac{\alpha_k}{2}$.
 We have from \eqref{concVn},
    \begin{equation*}
 \frac{1}{2}
 \le \mathbb{P}\left[V_n\le\E{V_n}+\sqrt{n\log2}\right]
 \le \sum_{i=1}^{n}\mathbb{P}\left[U_{[1,i],[1,n-i]}\le\E{V_n}+\sqrt{n\log2}\right]\,.
    \end{equation*}
 Hence, there exists $i$ such that
    \begin{equation*}
 \mathbb{P}\left[U_{[1,i],[1,n-i]}\le\E{V_n}+\sqrt{n\log2}\right]
 \ge \frac{1}{2n}\,.
    \end{equation*}
 Since $U_{[i+1,n],[n-i+1,n]}$ has the same distribution as
    $U_{[1,i],[1,n-i]}$, and they are independent of each others, we have
    \begin{align*}
 \frac{1}{4n^2}
 &\le\mathbb{P}\left[U_{[1,i],[1,n-i]}\le\E{V_n}+\sqrt{n\log2}\right]\cdot\mathbb{P}\left[U_{[i+1,n],[n-i+1,n]}\le\E{V_n}+\sqrt{n\log2}\right]\\
 &=\mathbb{P}\left[U_{[1,i],[1,n-i]}\le\E{V_n}+\sqrt{n\log2}\text{ and }U_{[i+1,n],[n-i+1,n]}\le\E{V_n}+\sqrt{n\log2}\right]\\
        &\le\mathbb{P}\left[U_{[1,n],[1,n]}\le2\E{V_n}+2\sqrt{n\log2}\right]
    \end{align*}
 From \eqref{concUn},
    \begin{equation*}
 \mathbb{P}\left[U_{n}\le\E{U_n}-\sqrt{n\log (4(n+1)^2)}\right]\le \frac{1}{4(n+1)^2}
 < \mathbb{P}\left[U_{n}\le2\E{V_n}+2\sqrt{n\log2}\right]\,,
    \end{equation*}
 which implies
    \begin{equation*}
 \E{U_n}-\sqrt{n\log (4(n+1)^2)}\le2\E{V_n}+2\sqrt{n\log2}\,.
    \end{equation*}
 This implies 
    $\E{U_n}\le2\E{V_n} + o(n)$.
 Taking the limit as $n$ goes to infinity, we get
    $\alpha_k\le 2\alpha'_k$ which concludes the proof.
\end{proof}

\subsection{Bounds on \texorpdfstring{$\alpha'_k$}{}}
For any $n\in\mathbb{N}$ and, $s,t\in \mathcal{A}^*$, we let 
$$V_n(s,t)= \mathbb{E}\left[\min_{\substack{i+j=n\\i,j\ge0 }} d_e(sX_1X_2\ldots X_i, tY_1Y_2\ldots Y_j)\right]\,.$$
As $n$ goes to infinity, the impact of the fixed prefixes $s$ and $t$ is negligible and the following Lemma is immediate.
\begin{lemma}\label{equalitywithVn}
 For any fixed $s,t\in \mathcal{A}^*$, we have 
    $$\alpha'_k = \lim\limits_{n\rightarrow\infty}\frac{V_n(s,t)}{n}\,.$$
\end{lemma}
We are now ready to bound $\alpha'_k$ using this second expression.
We let $\vec{V_n}$ be the vector whose components are the $V_n(s,t)$ where $s$ and $t$ range over all strings of length $h$ for a given $h$. Using the dynamic programming given earlier we provide upper bounds on the components of $\vec{V_n}$ based on $\vec{V_{n-1}} $ and $\vec{V_{n-2}}$.

\begin{fact}
 Let $a\in \mathcal{A}$ and $s,t\in \mathcal{A}^*$, then
    \begin{equation}\label{recSameLetter}
 V_n(as,at) \le\frac{1}{|\mathcal{A}|^2}\sum_{c,c'\in \mathcal{A}}V_{n-2}(sc,tc')\,.
    \end{equation}
\end{fact}
\begin{proof}
 We have
    \begin{align*}
 V_n(as,at) 
    &= \mathbb{E}\left[\min_{\substack{i+j=n\\i,j\ge0 }} d_e(asx_1x_2\ldots x_i, aty_1y_2\ldots y_j)\right]\\
    &= \mathbb{E}\left[\min_{\substack{i+j=n\\i,j\ge0 }} d_e(sx_1x_2\ldots x_i, ty_1y_2\ldots y_j)\right]\\
    &\le \mathbb{E}\left[\min_{\substack{i+j=n\\i,j\ge1 }} d_e(sx_1x_2\ldots x_i, ty_1y_2\ldots y_j)\right]\\
    &= \frac{1}{|\mathcal{A}|^2}\sum_{c,c'\in\mathcal{A}}\mathbb{E}\left[ \min_{\substack{i+j=n\\i,j\ge1 }} d_e(scx_2\ldots x_i, tc'y_2\ldots y_j)\right]\\
    &= \frac{1}{|\mathcal{A}|^2}\sum_{c,c'\in\mathcal{A}}\mathbb{E}\left[ \min_{\substack{i+j=n-2\\i,j\ge0 }} d_e(scx_1\ldots x_i, tc'y_1\ldots y_j)\right]\\
    &= \frac{1}{|\mathcal{A}|^2}\sum_{c,c'\in \mathcal{A}}V_{n-2}(sc,tc')\,,
    \end{align*}
 as desired.
\end{proof}

\begin{fact}
 Let $a,b\in \mathcal{A}$ such that $a\not=b$ and $s,t\in \mathcal{A}^*$, then
    \begin{align*}\label{recDiffLetter}
 V_n(as,bt)
    &\le 
    1+\begin{cases} 
 \frac{1}{|\mathcal{A}|}\sum_{c\in\mathcal{A}}V_{n-1}(sc,bt)  \\
 \frac{1}{|\mathcal{A}|}\sum_{c\in\mathcal{A}}V_{n-1}(as,tc) \\
 \frac{1}{|\mathcal{A}|^2}\sum_{c,c'\in\mathcal{A}}V_{n-2}(sc,tc')
    \end{cases}\,.
    \end{align*}
\end{fact}
\begin{proof}
 We have    
    \begin{align*}
 V_n(as,bt)
    &= \mathbb{E}\left[\min_{\substack{i+j=n\\i,j\ge0 }} d_e(asx_1x_2\ldots x_i, bty_1y_2\ldots y_j)\right]\\
    &= \mathbb{E}\left[\min_{\substack{i+j=n\\i,j\ge0 }}\left( 1+\min \begin{cases} 
 d_e(sx_1x_2\ldots x_i, bty_1y_2\ldots y_j)  \\
 d_e(asx_1x_2\ldots x_i, ty_1y_2\ldots y_j) \\
 d_e(sx_1x_2\ldots x_i, ty_1y_2\ldots y_j)
    \end{cases}\right)\right]\\
    &\le 1+\min \begin{cases}
 \mathbb{E}\left[\min_{\substack{i+j=n\\i,j\ge0 }}d_e(sx_1x_2\ldots x_i, bty_1y_2\ldots y_j)\right]  \\
 \mathbb{E}\left[\min_{\substack{i+j=n\\i,j\ge0 }}d_e(asx_1x_2\ldots x_i, ty_1y_2\ldots y_j)\right] \\
 \mathbb{E}\left[\min_{\substack{i+j=n\\i,j\ge0 }}d_e(sx_1x_2\ldots x_i, ty_1y_2\ldots y_j)\right]
    \end{cases}\,.
    \end{align*}
We upper bound the first term of this minimum as follows,
 \begin{align*}
 \mathbb{E}\left[\min_{\substack{i+j=n\\i,j\ge0 }}d_e(sx_1x_2\ldots x_i, bty_1y_2\ldots y_j)\right] 
    &\le\mathbb{E}\left[\min_{\substack{i+j=n\\i\ge1,j\ge0 }}d_e(sx_1x_2\ldots x_i, bty_1y_2\ldots y_j)\right]\\
    &=  \frac{1}{|\mathcal{A}|}\sum_{c\in\mathcal{A}}\mathbb{E}\left[\min_{\substack{i+j=n\\i\ge1,j\ge0 }}d_e(scx_2\ldots x_i, bty_1y_2\ldots y_j)\right]\\
    &=  \frac{1}{|\mathcal{A}|}\sum_{c\in\mathcal{A}}\mathbb{E}\left[\min_{\substack{i+j=n-1\\i\ge0,j\ge0 }}d_e(scx_1\ldots x_i, bty_1y_2\ldots y_j)\right]\\
    &=  \frac{1}{|\mathcal{A}|}\sum_{c\in\mathcal{A}}V_{n-1}(sc,bt)\,.
 \end{align*}
Similar computations for the two other terms of this minimum yields
\begin{equation*}
 V_n(as,bt)
 \le  1+ 
    \begin{cases} 
 \frac{1}{|\mathcal{A}|}\sum_{c\in\mathcal{A}}V_{n-1}(sc,bt)  \\
 \frac{1}{|\mathcal{A}|}\sum_{c\in\mathcal{A}}V_{n-1}(as,tc) \\
 \frac{1}{|\mathcal{A}|^2}\sum_{c,c'\in\mathcal{A}}V_{n-2}(sc,tc')
    \end{cases}\,,
\end{equation*}
 as desired.
\end{proof}

Let $T$ be the function given by these bounds such that \begin{equation}\label{boundOnVector}
 \vec{V_{n}} \le T(\vec{V_{n-1}},\vec{V_{n-2}})\,,
\end{equation}
where $\ge$ is taken component-wise. Moreover, $T$ is \textit{translationally invariant}, that is, for all real $r$, 
\begin{equation}\label{transInv}
 T(\vec{v}+r,\vec{v'}+r)=T(v,v')+r\,,
\end{equation}
where the sum of a vector with a real is taken component-wise. Finally, $T$ is also \emph{monotonic}, that is,
\begin{equation} \label{monotonic}
\vec{u}  \le\vec{u'} 
\text{ and }
\vec{v}  \le\vec{v'} 
\implies T(u,v) \le T(u',v')\,.
\end{equation}

We are now ready to prove the main Lemma for the computation of $\alpha'_k$.
\begin{lemma}
If there exists a vector $\vec{v}$ and a real $r$ such that
\begin{equation}\label{changeByR}
T(\vec{v},\vec{v}-r)\le \vec{v}+r\,,
\end{equation}
then there exists a real $c$ such that
\begin{equation}\label{rIsTheValue}
    \forall n\ge0, \vec{V_n} \le \vec{v}+nr+c\,.
\end{equation}
In particular, this implies that
$\alpha'_k\le r$.
\end{lemma}
\begin{proof}
We take any $c$ large enough such that \eqref{rIsTheValue} holds for $n\in\{0,1\}$. We proceed by induction on $n$. The base case is covered by the choice of $c$. For the inductive step, assume that \eqref{rIsTheValue} holds for all $i< n$. 
Then using \eqref{boundOnVector}, the induction hypothesis, \eqref{monotonic}, \eqref{transInv} and \eqref{changeByR}, in this order, we obtain
\begin{align*}
 \vec{V_n}
    &\le T(\vec{V_{n-1}},\vec{V_{n-2}})\\
    &\le T(\vec{v}+(n-1)r+c,\vec{v}+(n-2)r+c)\\
    &\le T(\vec{v},\vec{v}-r)+(n-1)r+c\\
    &\le \vec{v}+r+(n-1)r+c= \vec{v}+nr+c\,,
\end{align*}
which concludes our induction.
\end{proof}

\subsection{Implementation and bounds}

If we can find a vector $\vec{v}$ and real $r$ satisfying \eqref{changeByR}, this implies $\alpha_k\le 2r$. We use a computer program to verify the existence of such a vector. 
To find $\vec{v}$ and $r$, we simply start with two vectors $\vec{w_0}=\vec{w_1}=0$, and we iterate $\vec{w_n}=T(\vec{w_{n-1}},\vec{w_{n-2}})$. It seems to be the case that $\vec{w_n}-\vec{w_{n-1}}$ converges to a uniform vector, but we do not need to prove it.
We simply take $\vec{v}=w_n$ for some $n$ large enough, and we take $r$ to be the largest component of $w_n-w_{n-1}+\varepsilon$. We then check that \eqref{changeByR} is satisfied with this choice of  $\vec{v}$ and $r$. 

The implementation relies on fixed point arithmetic (that is, we fix some positive integer $p$, and we work with numbers of the form $a/p$ where $a\in\mathbb{N}$). We do all the rounding in the appropriate direction and in particular for the final step, we can formally verify with exact computation that \eqref{changeByR} holds with the choice of $\vec{v}$ and $r$.

In the implementation\footnote{The code is available on \url{https://gite.lirmm.fr/mrosenfeld/expected-edit-distance}.}, every pair of strings $(u,v)$ is manipulated as the corresponding number $uv$ written in base $k$ (every considered pair is stored in a $64$ bits integers). We also use the fact that a permutation of the alphabet does not change the result. That is, for every function $f$ that permutes the letters of the alphabet and for all strings $u,v$ of length $h$, $V_n(u,v)= V_n(f(u),f(v))$. For every pair $u,v$, we consider the $f$ that minimizes the string $uv$ for the lexicographic order (and we call $f(uv)$ the \emph{normalization} of  $uv$). This allows us to divide the size of the considered vector by $\sim k!$. However, it can be costly to recompute the renaming of the strings, so we provide three different versions of the code that provide different optimizations that are better for different values of $k$.

\paragraph{The binary case.} A first version was written specifically for the binary alphabet. In this case, exchanging the letters of the string for the normalization is simply taking the bitwise negation of the string, and we can use bit-shifts and other bit manipulation to produce a faster code. In this code, the transformation $T$ is computed on the fly for every iteration, and we only store the vectors. The results obtained with this program are given in Table \ref{tableBinaryEdit}. The source code is in the file \path{EditDistance_binary_alphabet.cpp} from the GIT repository. This is the only version that uses parallelization (the optimal number of cores seems to be 2 as using more cores seems to increase considerably the number of cache misses).

We obtain $\alpha_2\le0.315514$. In \cite{Lueker2009May}, the author uses other symmetries that one might be able to adapt here to push the computations to $h=18$ instead of $h=17$ with the same machine (which, based on the difference between $h=16$ and $h=17$, would not improve the bound by much). Let us mention that the results obtained by Lueker have been very recently improved by optimizing the implementation and some of the techniques used might be applicable in our case as well \cite{Heineman2024Jul}.
\begin{table}[!h]
    \begin{center}
    \begin{tabular}{ |c|c|c|c|c|c| } 
        \hline
 h  & Upper bound & Memory consumption & Computation time\\\hline\hline
 14 & 0.319052 & 1.6Gb & 1m19s\\\hline
 15 & 0.317752 & 6.2Gb & 6m\\\hline
 16 & 0.31658 & 25Gb & 17m\\\hline
 17 & 0.315514 & 100Gb & 1h32m\\\hline
    \end{tabular}
    \end{center}
    \caption{Our results for the binary case. The number of iterations of $T$ before choosing  $\vec{v}=w_n$ is $n=150$. We used a laptop for $h\le15$ and a more powerful computer for $h>15$. All the times here (and in futur tables) are purely indicative and do not constitute a proper benchmarking.}\label{tableBinaryEdit}
\end{table}

\paragraph{The general case.} The second version of the code precomputes the transformation $T$ and stores it. In the general setting of an alphabet of size $k$ the normalization is more costly to compute, so we only compute it once for every pair of strings. Storing $T$ increases the memory consumption (by a factor of order $k$), but saves a lot of computation time.
The results obtained with this program are given in Table \ref{tableSmallAlphabetsEdit}, and the source code is in the file \path{EditDistance_general_case.cpp} from the GIT repository.
\begin{table}[!h]
    \begin{center}
    \begin{tabular}{ |c|c|c|c|c|c| } 
        \hline
 k  & Upper bound & $h$ used & Memory consumption & Computation time & Iterations \\\hline\hline
 3  & 0.47276 & 9 & 20Gb & 17m & 100 iterations\\\hline
 4  & 0.56578 & 8 & 58Gb & 1h06 & 100 iterations\\\hline
 5  & 0.6325 & 7 & 28Gb & 32m & 100 iterations\\\hline
    \end{tabular}
    \end{center}
    \caption{Our results for the small alphabets using the second version of the code.}\label{tableSmallAlphabetsEdit}
\end{table}

\paragraph{Large alphabets.} In the third version we slightly change the data structure used to store the transformation $T$. This data structure uses maps instead of vectors and is more costly for small alphabets, but the cost is now linear in $\max(k,2h)$ instead of $k$ which means that for fixed $h$ we can compute a lower bound for very large values of $k$.
The results obtained with this program are given in Table \ref{tableLargeAlphabetsEdit}, and the source code is in the file \path{EditDistance_rally_large_alphabets.cpp} from the GIT repository.

\begin{table}[!h]
    \begin{center}
    \begin{tabular}{ |c|c|c|c|c|c|c } 
        \hline
 k  & Upper bound & $h$ used & Memory consumption & Computation time & Iterations \\\hline\hline
 3  & 0.47626 & 8 & 6.4Gb & 1m28s & 50 iterations\\\hline
 4  & 0.57552 & 6 & 0.9Gb & 14s & 50 iterations\\\hline
 5  & 0.63792 & 6 & 3.4Gb & 1m10s & 50 iterations\\\hline
 6  & 0.68424 & 6 & 7.2Gb & 2m46s & 50 iterations\\\hline
 7  & 0.72016 & 6 & 9.4Gb & 4m03s & 50 iterations\\\hline
 8  & 0.74896 & 6 & 10.3Gb & 4m23s & 50 iterations\\\hline
 9  & 0.77264 & 6 & 10.4Gb & 4m43s & 50 iterations\\\hline
 10  & 0.7925 & 6 & 10.5Gb & 5m14s & 50 iterations\\\hline
 11  & 0.8095 & 6 & 10.5Gb & 5m43s & 50 iterations\\\hline
 12  & 0.82432 & 6 & 10.5Gb & 6m15s & 50 iterations\\\hline
 13  & 0.83744 & 6 & 10.5Gb & 6m57s & 50 iterations\\\hline
 14  & 0.84646 & 6 & 10.5Gb & 9m45s & 150 iterations\\\hline
 15  & 0.85608 & 6 & 10.5Gb & 10m54s & 150 iterations\\\hline
 16  & 0.86462 & 6 & 10.5Gb & 10m54s & 150 iterations\\\hline
 17  & 0.87228 & 6 & 10.5Gb & 10m50s & 150 iterations\\\hline
 18  & 0.87916 & 6 & 10.5Gb & 10m47s & 150 iterations\\\hline
 19  & 0.88536 & 6 & 10.5Gb & 10m54s & 150 iterations\\\hline
 20  & 0.89102 & 6 & 10.5Gb & 10m49s & 150 iterations\\\hline
 21  & 0.89614 & 6 & 10.5Gb & 10m55s & 150 iterations\\\hline
 22  & 0.90084 & 6 & 10.5Gb & 10m47s & 150 iterations\\\hline
 23  & 0.90514 & 6 & 10.5Gb & 10m46s & 150 iterations\\\hline
 24  & 0.90912 & 6 & 10.5Gb & 10m35s & 150 iterations\\\hline
 25  & 0.91278  & 6 & 10.5Gb & 10m52s & 150 iterations\\\hline
 32  & 0.93228 & 6 & 10.5Gb & 10m44s & 150 iterations\\\hline
 100 & 0.97946 & 6 & 10.5Gb & 28m & 500 iterations\\\hline
 1000 & 0.9982 & 6 & 10.5Gb & 3h40m & 5000 iterations\\\hline
    \end{tabular}
    \end{center}
    \caption{Our results for larger alphabets using the third version of the code. The results are all obtained using a laptop.}\label{tableLargeAlphabetsEdit}
\end{table}

\section{Longest common subsequence}
The technique that we use to compute bounds on $\gamma_k$ is the same as in \cite{Lueker2009May}, and is moreover almost identical to what we did in the previous section. 
For the sake of completeness, we provide the definitions and the corresponding statement, but we do not repeat the proofs.

For any strings $u,v$, $\operatorname{LCS}(u,v)$ is the size of the longest common subsequence of $u$ and $v$.
Remember, that given two strings $u,v\in \A^*$, and letters $a,b\in \A$,
we have 
\begin{equation}\label{lcsdyn}
 \operatorname{LCS}(au,bv)=
\begin{cases} 
    1+\operatorname{LCS}(u,v) & \text{ if } a=b \\
 \max(\operatorname{LCS}(u,bv),\operatorname{LCS}(au,v)) & \text{ otherwise }
 \end{cases}
\end{equation}
and for all $u$, $\operatorname{LCS}(u,\varepsilon)=\operatorname{LCS}(\varepsilon,u)=0$, where $\varepsilon$ is the empty string.
For any strings $u,v,u',v'$, we have 
\begin{equation}\label{decompositionLCS}
 \operatorname{LCS}(uv,u'v')\ge \operatorname{LCS}(u,u')+\operatorname{LCS}(v,v')\,.
\end{equation} 

For the rest of this section, the $X_i$ and $Y_i$ are independent uniform random variables from $\A=\{1,\ldots, k\}$. 
For all $k$, the quantity $\gamma_k$ is defined as
\begin{equation*}
    \gamma_k = \lim_{n\rightarrow\infty}\frac{\E{\operatorname{LCS}(X_1\ldots X_n, Y_1\ldots Y_n)}}{n}\,.
\end{equation*}
Linearity of expectation and equation \eqref{decompositionLCS} imply that the sequence is subadditive which by Feketes's lemma implies that $\alpha_k$ is well-defined. 

For any $n\in\mathbb{N}$ and, $s,t\in \mathcal{A}^*$, we let 
$$W_n(s,t)= \mathbb{E}\left[\min_{\substack{i+j=n\\i,j\ge0 }} \operatorname{LCS}(sX_1X_2\ldots X_i, tY_1Y_2\ldots Y_j)\right]\,.$$
The following Lemma can be proven by following the same proof as for Lemma \ref{equalitywithVn}.
\begin{lemma}
 For any fixed $s,t\in \mathcal{A}^*$, we have 
    $$\gamma_k = \lim\limits_{n\rightarrow\infty}\frac{W_n(s,t)}{n}\,.$$
\end{lemma}

We let $\vec{W_n}$ be the vector whose components are the $W_n(s,t)$ where $s$ and $t$ range over all strings of length $h$ for a given $h$. Equation \eqref{lcsdyn} implies lower bounds on the components of $\vec{W_n}$ based on $\vec{W_{n-1}} $ and $\vec{W_{n-2}}$.
\begin{fact}
 Let $a\in \mathcal{A}$ and $s,t\in \mathcal{A}^*$, then
    \begin{equation}\label{recSameLetterLCS}
 W_n(as,at) \ge 1+\frac{1}{|\mathcal{A}|^2}\sum_{c,c'\in \mathcal{A}}W_{n-2}(sc,tc')\,.
    \end{equation}
\end{fact}

\begin{fact}
 Let $a,b\in \mathcal{A}$ such that $a\not=b$ and $s,t\in \mathcal{A}^*$, then
    \begin{equation}\label{recDiffLetterLCS}
 W_n(as,bt) \ge 
    \begin{cases} 
 \frac{1}{|\mathcal{A}|}\sum_{c\in\mathcal{A}}W_{n-1}(sc,bt)  \\
 \frac{1}{|\mathcal{A}|}\sum_{c\in\mathcal{A}}W_{n-1}(as,tc)
    \end{cases}
    \end{equation}
\end{fact}

Let $T$ be the function given by bounds \eqref{recSameLetterLCS} and \eqref{recDiffLetterLCS} such that 
\begin{equation*}
 \vec{W_{n}} \ge T(\vec{W_{n-1}},\vec{W_{n-2}})\,,
\end{equation*}
where $\ge$ is taken component-wise.
We are now ready to state the main Lemma behind the computation of the upper bounds for $\gamma_k$.
\begin{lemma}
If there exists a vector $\vec{w}$ and real $r$ such that
\begin{equation*}
T(\vec{w},\vec{w}-r)\ge \vec{w}+r\,,
\end{equation*}
then there exists a real $c$ such that
\begin{equation*}
    \forall n\ge0, \vec{W_n} \ge \vec{w}+nr+c\,.
\end{equation*}
In particular, this implies that
$\gamma_k\ge 2r$.
\end{lemma}

\paragraph{Implementation and results}
Our implementation is obtained by adapting the code used in the setting of the edit distance. There are a couple of lines to change for the recurrence relations and a few other lines to change because the direction of the inequality is different (in particular, the rounding for division is already as desired in this case while we had to be careful in the edit distance setting).
We have three different codes optimized for the binary case, the small alphabets and arbitrarily large alphabets. The results are respectively given in Table \ref{tableBinaryAlphabetsLCS}, Table \ref{tableSmallAlphabetsLCS}, and Table \ref{tableLargeAlphabetsLCS}.

\begin{table}[!h]
    \begin{center}
    \begin{tabular}{ |c|c|c|c|c|c| } 
        \hline
 h  & Upper bound & Memory consumption & Computation time\\\hline\hline
 15 & 0.78806 & 6.3Gb & 5m48s \\\hline
 16 & 0.78901 & 25Gb & 23m01s \\\hline
 17 & 0.789872 & 100Gb & 2h2m \\\hline
    \end{tabular}
    \end{center}
    \caption{Our results for the binary case. The number of iterations of $T$ before choosing $\vec{v}=\vec{W_n}$ is $n = 150$. Our result improves on the bound given by \cite{Lueker2009May}, but fail to improve on the bound given by \cite{Heineman2024Jul}.}\label{tableBinaryAlphabetsLCS}
\end{table}

\begin{table}[!h]
    \begin{center}
    \begin{tabular}{ |c|c|c|c|c|c| } 
        \hline
 k  & Lower bound & $h$ used & Memory consumption & Computation time & Iterations \\\hline\hline
 3 \tablefootnote{The discrepancy between $0.6821$ and the bound $0.682218$ from \cite{Heineman2024Jul} which is obtained with same algorithm and the same value $h$ can be explained by the choice of $p=100000$ and $\varepsilon=0.00005$ used for the exact computation of $v$ and $r$ in fixed arithmetic precision.}  & 0.6821 & 9 & 12.7Gb & 9m33s & 100 iterations\\\hline
 3  & 0.68422 & 10 & 114b & 2h44m29s& 100 iterations\\\hline
 4  & 0.61046 & 7 & 2.2Gb & 1m03 & 50 iterations\\\hline
 4  & 0.61422 & 8 & 35.2Gb & 29m & 100 iterations\\\hline
 5  & 0.56206 & 7 & 14.3Gb & 8m33s & 50 iterations\\\hline
    \end{tabular}
    \end{center}
    \caption{Our results for the small alphabets using the second version of the code.}\label{tableSmallAlphabetsLCS}
\end{table}

\begin{table}[!h]
    \begin{center}
    \begin{tabular}{ |c|c|c|c|c|c| } 
        \hline
 k  & Lower bound & $h$ used & Memory consumption & Computation time  \\\hline\hline
 3  & 0.67932 & 8 & 4.1Gb & 54s\\\hline
 4  & 0.61046 & 7 & 7.0Gb & 1m54s\\\hline
 5  & 0.55686 & 6 & 1.7Gb & 26s\\\hline
 6  & 0.51850 & 6 & 3.0Gb & 55s \\\hline
 7  & 0.48712 & 6 & 3.8Gb & 1m14s \\\hline
 8  & 0.46074 & 6 & 4.0Gb & 1m22s \\\hline
 9  & 0.43806 & 6 & 4.0Gb & 1m30s \\\hline
 10 & 0.41826 & 6 & 4.0Gb & 1m35s \\\hline
 11 & 0.40072 & 6 & 4.0Gb & 1m43s \\\hline
 12 & 0.38504 & 6 & 4.0Gb & 1m47s \\\hline
 13 & 0.37088 & 6 & 4.0Gb & 1m58s \\\hline
 14 & 0.35798 & 6 & 4.0Gb & 1m55s \\\hline
 15 & 0.34616 & 6 & 4.0Gb & 1m56s \\\hline
 16 & 0.33528 & 6 & 4.0Gb & 1m58s \\\hline
 17 & 0.32518 & 6 & 4.0Gb & 1m56s \\\hline
 18 & 0.31580 & 6 & 4.0Gb & 2m02s \\\hline
 19 & 0.30702 & 6 & 4.0Gb & 1m55s \\\hline
 20 & 0.29880 & 6 & 4.0Gb & 1m57s \\\hline
 21 & 0.29106 & 6 & 4.0Gb & 1m55s \\\hline
 22 & 0.28376 & 6 & 4.0Gb & 1m56s \\\hline
 23 & 0.27686 & 6 & 4.0Gb & 1m57s \\\hline
 24 & 0.27032 & 6 & 4.0Gb & 2m5s  \\\hline
 25 & 0.26412 & 6 & 4.0Gb & 1m59s \\\hline
 50 & 0.16930 & 6 & 4.0Gb & 1m52s \\\hline
 100 & 0.0991 & 6 & 4.0Gb & 1m59s \\\hline
 1000 &0.01164 & 6 & 4.0Gb & 2m6s \\\hline
    \end{tabular}
    \end{center}
    \caption{Our results for larger alphabets using the third version of the code. The results are all obtained using a laptop. The number of iterations of $T$ before choosing $\vec{v}=\vec{W_n}$ is $n = 50$.}\label{tableLargeAlphabetsLCS}
\end{table}

The computation of upper bounds on $\gamma_k$ for really large $k$ might be of limited interest. Indeed, general lower bounds given by simple explicit closed-form formulas will be better for large $k$. 
If we limit ourselves to windows of size $h=6$, for $k$ large enough all the letters in the window are different with really high probability, and it is not hard to verify that the lower bounds that this technique can establish then behave in $O(1/k)$. On the other hand, it is known that $\gamma_k$ behaves in $2/\sqrt{k}+o(1/\sqrt{k})$ \cite{Kiwi2005Nov}. It is claimed in \cite{Baeza-Yates1999Aug} that \cite{chvatal} contains a proof of $\gamma_k\ge 1/\sqrt{k}$, but the only result in \cite{chvatal} that resembles this is $\gamma_k\ge \frac{2k^2}{k^3+2k-1}$. An explicit lower bound in $\Theta(\sqrt{k})$ can probably be adapted from the asymptotic lower bound of \cite{Deken1979Jan}. Assuming that the claimed lower bound of $\gamma_k\ge 1/\sqrt{k}$ is true, it is a better bound that what can be done with our software on modern machines for any $k\ge k_0$ for some $k_0$ between $50$ and $100$.

\printbibliography
\end{document}